\documentclass[a4paper,10pt]{amsart}
\usepackage{amssymb,amsfonts,amsthm}
\usepackage{amstext}
\usepackage{comment}
\usepackage{amsmath}
\usepackage{hyperref}  

\usepackage{color}
\usepackage[latin1]{inputenc}
\usepackage[active]{srcltx}
\usepackage{mathrsfs} 
\usepackage{comment} 
\usepackage{graphicx}
\usepackage{amssymb,amsfonts,amstext, amsmath}
\usepackage[latin1]{inputenc}
\usepackage{tensor}
\usepackage{color}
\usepackage{mathptmx}

\newtheorem{thm}{Theorem}[section]
\newtheorem{cor}[thm]{Corollary}
\newtheorem{lem}[thm]{Lemma} 

\newtheorem{ex}[thm]{Example}

\newtheorem{defn}[thm]{Definition}
\newtheorem{rem}[thm]{Remark}

\DeclareMathAlphabet\EuScript{U}{eus}{m}{n}
\SetMathAlphabet\EuScript{bold}{U}{eus}{b}{n}

\newtheorem*{rep@theorem}{\rep@title}
\newcommand{\newreptheorem}[2]{%
	\newenvironment{rep#1}[1]{%
		\def\rep@title{#2 \ref{##1}}%
		\begin{rep@theorem}}%
		{\end{rep@theorem}}}
\makeatother

\usepackage{ dsfont }

\usepackage{chngcntr} 

\usepackage{mathrsfs} 

\usepackage{comment} 

\usepackage{hyperref}  
\begin{document}

	\title{
		{\textbf{Positive definite matrix valued kernels and their scalar valued  projections: counterexamples}} \vspace{-4pt}
	}

	\author[1]{J. C. Guella}
	\email{jcguella@unicamp.br}
	\address{Institute of Mathematics, Statistics and Scientific Computing, University of Campinas}

 

 	\maketitle
  
	\begin{center}
		\parbox{13 cm}{{\small {\bf Abstract:} In this paper we show that the strictly positive definite matrix valued isotropic kernels in the circle and the real dot product kernels in Euclidean spaces are not  well behaved with respect to its scalar valued projections. We generalize the counterexamples that we obtained to an abstract setting by using the concepts of  unitarily invariant kernels and adjointly invariant kernels, provided the existence of  an aperiodic invariant function.  }}
	\end{center}

	\tableofcontents

	\section{Introduction}

	A matrix valued kernel $K:X \times X \to M_{\ell}(\mathbb{C})$ is positive definite if for every  finite quantity of distinct points $x_1,x_2, \ldots, x_n$ in $X$ and vectors $v_1,v_2,$ $\ldots, v_n\in \mathbb{C}^{\ell}$, we have
	$$
	\sum_{\mu, \nu=1}^n  \langle K(x_\mu ,x_\nu ) v_{\mu}, v_{\nu}\rangle   \geq 0.
	$$
	In addition, if the  above  inequalities are strict whenever at least one of the vectors $v_\mu$ is nonzero, then the kernel is termed  strictly positive definite.
	
	When $\ell = 1$, the previous definition is the standard notion of positive definite kernel (strictly). A very simple connection between those two concepts is obtained by the   scalar valued projections of the kernel, which are the scalar valued kernels $K_{v}: X \times X \to \mathbb{C}$, $v \in  \mathbb{C}^{\ell} \setminus\{ 0\}$, given by
	\begin{equation}\label{eq1}
		K_{v}(x,y)=\langle K(x ,y ) v, v\rangle, \quad x,y \in X.
	\end{equation}
	
	If the matrix valued kernel $K:X \times X \to M_{\ell}(\mathbb{C})$ is a positive definite (strictly), then all of its scalar valued projections kernels are positive definite (strictly). However, the converse of this property is not true in general. For instance,  given two functions $f:X \to \mathbb{R}$ and $h:X \to \mathbb{R}$, all scalar valued projections of the kernel
	$$
	K(x,y) =	\left(\begin{array}{cc} 2f(x)f(y) & f(x)h(y) + h(x)f(y) \\ f(x)h(y) + h(x)f(y) & 2h(x)h(y) \end{array}\right),
	$$
	are positive definite. Suitable choice of the functions $f$ and $h$ lead to a counterexample, like $f(x_{1})=1$, $f(x_{2})=1$, $h(x_{1})=1$ and $h(x_{2})=2$, for  arbitrary distinct points $x_{1} , x_{2}$ in $X$. Note that if the set $X$ has a topology for which the space real valued continuous functions defined on $X$, that is $C(X, \mathbb{R})$,  has dimension at least two, the above method generates a counterexample for the converse of the scalar valued projections  when the kernel $K$ is continuous.

	However, for some specific kernels, the fact that all scalar valued projections of a matrix valued kernel are positive definite (strictly) implies that the matrix valued kernel is positive definite (strictly):\\

	(1)\textbf{Isotropic kernels in real spheres} $X=S^d$, the unit sphere in $\mathbb{R}^{d+1}$, $K$ continuous and also isotropic in the sense that
	$$
	K(Qx,Qy)=K(x,y), \quad x,y \in S^d; Q\in \mathcal{O}(d+1),
	$$
	where $\mathcal{O}(d+1)$ is the set of all orthogonal transformations on $\mathbb{R}^{d+1}$.\ The characterization of the positive definite scalar valued  kernels fulfilling this condition was achieved in \cite{schoen}, while the strict case was proved in \cite{chen} ($d\geq 2$)  and \cite{spds1} ($d=1$). The positive definite matrix valued case traces back to \cite{yaglom}, while the strict case was proved in \cite{porcu} with the exception of the case $d=1$, which remained open. With the exception of strict kernels in $S^{1}$, all of them are well behaved  with respect to the scalar valued projections. \\
	There is also the limiting case, of kernels defined on $S^{\infty}$,  the unit sphere of an infinity dimensional real Hilbert space, $K$ is continuous and the isotropy is the invariance of the kernel for all linear isometries in it. The previous references contain the analysis for this case and it is well behaved with respect to the scalar valued projections property.   
	
	(2)\textbf{Isotropic kernels in two-point compact homogeneous spaces} $X=M^d$, where $M^{d}$ is a compact connected two point homogeneous space, but not a real sphere in $\mathbb{R}^{d+1}$, $K$ continuous and also isotropic in the sense that
	$$
	K(Qx,Qy)=K(x,y), \quad x,y \in S^d; Q \in ISO(M^{d})
	$$
	where $ISO(M^{d})$ is the set of all isometries in $M^{d}$. The two-point compact homogeneous spaces were classified in \cite{homo}, while the characterization of the positive definite isotropic kernels were characterized in \cite{gangolli}. The strictly positive definite case was proved in \cite{spdhomo}, and the matrix valued results were obtained in \cite{BONFIM2016237}. Unlike the real sphere, the characterizations for the infinite dimensional case (for the real, complex and quaternionic projective spaces) came as a corollary of a different  paper \cite{limit}. Although there is no paper that explicitly describes the matrix valued case in the infinite dimensional setting, it is well behaved with respect to the scalar valued projections as well as the other cases.

	(3)\textbf{Radial kernels} $X=\mathbb{R}^{d}$, $K$ continuous and also radial in the sense that
	$$
	K(Qx +z,Qy+z)=K(x,y), \quad x,y,z \in \mathbb{R}^d; Q\in \mathcal{O}(d).
	$$
	The characterization of the positive definite scalar valued  kernels fulfilling this condition was achieved in \cite{schoen}, while the strict case was proved in \cite{sun} ($d\geq 2$). The positive definite matrix valued case was characterized in  \cite{covariancema}, while the strict case was proved in \cite{jeanrad} with the exception of the case $d=1$, which remained open. \\
	Like the real sphere, the previous papers contain proofs for the infinite dimensional case.\\
	With the exception of strict  kernels in $R^{1}$, all of them are well behaved  with respect to the scalar valued projections. 
	
	(4)\textbf{ Translation invariant kernels} $X=\mathbb{R}^{d}$, $K$ is continuous and translation invariant, in the sense that
	$$
	K(x+z,y+z)= K(x,y), \quad x,y,z \in \mathbb{R}^{d},
	$$
	The characterization of all positive definite kernels that fulfill this property is in the classical paper \cite{bochner}. Only sufficient conditions of when this kernel is strict are known. An operator valued version of those kernels can be found in  \cite{neb}. Only the positive definite  case is well behaved with respect to the scalar valued projections.

	(5)\textbf{ Real Dot product kernels  } $X=\mathcal{H}$, a real  Hilbert space,  $K$ is continuous and adjointly invariant, in the sense that
	$$
	K(Ax,y)= K(x,A^{t}y), \quad x,y \in \mathcal{H},\quad  A \in \mathcal{L}(\mathcal{H})  
	$$
	where $A^{t}$ is the adjoint operator of $A$.  If a kernel fulfills this property then there exists a function $h: \mathbb{R}  \to \mathbb{R}$ such that $K(x,y)= h(\langle x,y \rangle )$ (we prove this affirmation in Lemma \ref{dot1}). The characterization of which functions $h$ generates a positive definite kernel is obtained as a  corollary of a result in \cite{schoen} when $\dim(\mathcal{H})= \infty$, and generalized to  when $\dim(\mathcal{H})\geq 2$ in \cite{FangyanHongwei2005}. The characterization for when  $\dim(\mathcal{H})=1$ can be obtained as Corollary of Remark 3.9 in \cite{berg0}  on page 161.  The strict case was proved in \cite{pinkus} for $\dim(\mathcal{H}) \geq 2$. We are not aware of the strict  characterization for when $\dim(\mathcal{H})=1$. Only the positive definite case is well behaved with respect to the scalar valued projections.

	In the following example we close the gaps  in the families $(1)$, $(4)$ and $(5)$ by providing an example of a positive definite matrix valued kernel that is not strictly positive definite but whose all scalar valued projections are strictly positive definite:

	\begin{ex}\label{examplekern} Consider the kernels
		$$
		((\cos\theta, \sin \theta),(\cos \vartheta, \sin \vartheta)) \in S^{1} \times S^{1} \to  \left[\begin{array}{cc}
			e^{\cos(\theta - \vartheta)} & e^{\cos(\theta + \rho - \vartheta)} \\
			e^{\cos(\theta  - \vartheta- \rho)}  & e^{\cos(\theta - \vartheta)}
		\end{array}\right], \quad \rho \in [0, 2\pi), \rho \notin \mathbb{Q}\pi   .
		$$		 
		$$
		(x,y) \in \mathbb{R}^{m} \times \mathbb{R}^{m} \to  \left[\begin{array}{cc}
			e^{- \sigma\| x - y \|^{2} } &e^{- \sigma\| x+z - y \|^{2} } \\
			e^{- \sigma\| x - y -z \|^{2} }  & e^{- \sigma\| x - y \|^{2} } 
		\end{array}\right], \quad z \in  \mathbb{R}^{m} \setminus{\{0\}}.
		$$
		
		$$
		(x,y) \in \mathbb{R}^{m} \times \mathbb{R}^{m} \to  \left[\begin{array}{cc}
			e^{r^{2}\langle x,y \rangle}+1 & e^{r\langle x,y \rangle} \\
			e^{r\langle x,y \rangle}  & e^{\langle x,y \rangle} +  1
		\end{array}\right], \quad r \in  \mathbb{R} \setminus{\{0, 1, -1\}}.
		$$
		which are, respectively, an isotropic kernel in $S^{1}$, a translation invariant kernel in $\mathbb{R}^{m}$ and a real dot product kernel in $\mathbb{R}^{m}$. They are positive definite but not strict (which is a consequence of Theorem \ref{construcao} ) and  all scalar valued projections are strictly positive definite. 	
	\end{ex}

	In Section \ref{Unitarily invariant kernels} we generalize the first two examples to a broader setting, of kernels unitarily invariant by a family (more precisely a semigroup) of functions, where the criteria is based on the existence of an aperiodic function in the center of the semigroup. In Section  \ref{Adjointly invariant kernels} we obtain a similar result, by generalizing the third example  for kernels that are adjointly invariant by a semigroup of functions with involution, provided the existence  of an aperiodic function in the center of the semigroup. 
	
	We reemphasize that all known  characterizations of matrix valued strictly positive definite kernels with a symmetry (either unitarily or adjointly), satisfy the scalar valued projections. Hence, for the examples presented, a new and completely different method will be required to characterize  those matrix valued strictly positive definite kernels.

	\section{Unitarily invariant kernels}\label{Unitarily invariant kernels}
	We recall that a semigroup $(S, \circ)$ is a set $S$ together with an operation $\circ :  S \times S \to S$ that is associative, meaning $((a \circ b) \circ c) = (a \circ (b \circ c))$, for every $a,b,c \in S$. A semigroup $(S, \circ)$ has a unity if there exist $e \in S$ for which $e \circ a= a \circ e=a$ for all $a \in S$.

	The classes of kernels $(1), (2), (3) (4)$ have something in common, those kernels satisfy the following property:

	\begin{defn}\label{unitarily}Let $X$ be a topological space and $S=\{\phi_{\lambda}:X \to X , \lambda \in \Lambda\}$ be a family of continuous functions. We say that a continuous  matrix valued kernel   $K: X \times X \to M_{\ell}(\mathbb{C})$ is unitarily invariant the  family of functions $S$ if
		$$
		K(\phi_{\lambda}(x), \phi_{\lambda}(y))= K(x,y)
		$$
		for all $x,y \in X$ and $\lambda \in \Lambda$ (or $\phi_{\lambda} \in S$).
	\end{defn}
	We are assuming that the space $X$ has a topology and the kernel $K$ is continuous, only by its importance rather than a condition itself. Although at the above definition we did not make any assumption on the set $S$, without loss of generality we can assume that it is a semigroup  of continuous functions with a unity. Indeed the identity function $i:X \to X$ satisfies the required symmetry relation, and given two functions $\phi_{1}, \phi_{2} \in S$, then $\phi_{2}\phi_{1}$ satisfies
	$$
	K(x,y)=K(\phi_{1}(x), \phi_{1}(y))= K(\phi_{2}(\phi_{1}(x)), \phi_{2}(\phi_{1}(y))) , \quad x,y \in X. 
	$$
	
	If a  continuous matrix valued kernel $K$ is unitarily invariant by a semigroup of continuous functions with unity $S$ and is positive definite, we use the notation $K \in P(X,S, \mathbb{C}^{\ell})$. If in addition the kernel is strictly positive definite we use the notation $K \in P^{+}(X,S, \mathbb{C}^{\ell})$. Similarly, if  all of the scalar valued projections of a continuous matrix valued kernel   $K$ that is unitarily invariant by a semigroup of continuous functions  with unity $S$ are scalar valued positive definite kernels, we use the notation $K \in P_{proj}(X,S, \mathbb{C}^{\ell})$. Likewise if in addition all of the scalar valued projections of the kernel $K$ are scalar valued strictly positive definite kernels, we use the notation $K \in P^{+}_{proj}(X,S, \mathbb{C}^{\ell})$. When $\ell =1$ we use the simplified notation $P(X,S)$ and $P^{+}(X,S)$.
	
	We reemphasize that the inclusions
	\begin{equation}
		P(X,S, \mathbb{C}^{\ell}) \subset P_{proj}(X,S, \mathbb{C}^{\ell})
	\end{equation}
	\begin{equation}\label{subset}
		P^{+}(X,S, \mathbb{C}^{\ell}) \subset P^{+}_{proj}(X,S, \mathbb{C}^{\ell})
	\end{equation}
	always holds and that under  the conditions of Theorem \ref{principal} we show that the inclusion $P^{+}(X,S, \mathbb{C}^{\ell}) \subset P^{+}_{proj}(X,S, \mathbb{C}^{\ell})$ is strict.

	The following simple Lemma explains why we focus only on  $M_{2}(\mathbb{C})$ valued kernels.

	\begin{lem}\label{just2} Let $X$ be a topological space, $S$ a semigroup of continuous functions on $X$. Then  $P^{+}_{proj}(X,S,\mathbb{C}^{\ell})=P^{+}(X,S,\mathbb{C}^{\ell}) $ for some $\ell \geq 2$ if and only if  $P^{+}_{proj}(X,S,\mathbb{C}^{m})=P^{+}(X,S,\mathbb{C}^{m}) $  for every $2\leq m \leq \ell$.
	\end{lem}
	
	\begin{proof}
		The converse is immediate. If it holds for an $\ell$, then for any $2\leq m \leq \ell$ and kernel $K$ in $P^{+}_{proj}(X,S,\mathbb{C}^{m})$ we extend it to have $M_{\ell}(\mathbb{C})$ values by adding $0$ in the remaining  $\ell^{2}-m^{2}  -(\ell - m)$ off diagonal coordinates and on the remaining $ \ell-m$ diagonal coordinates we put  an arbitrary $P^{+}(X,S)$ kernel (like $\langle K e_{1}, e_{1}\rangle $). This new kernel is  in  $P^{+}_{proj}(X,S,\mathbb{C}^{\ell})$, hence in $P^{+}(X,S,\mathbb{C}^{\ell})$, and from that we obtain that the original kernel $K$ is in $P^{+}(X,S,\mathbb{C}^{m})$.
	\end{proof}

	Another  simple statement that we use, which is independent of the setting of kernels unitarily invariant by a semigroup of functions, but will simplify the construction of the counterexample is the following

	\begin{lem}	
		Let $K : X \times X  \to M_{2}(\mathbb{C})$ be a kernel. Then $K$ is (strictly) positive definite if and only if for every finite quantity of distinct points $x_{1}, x_{2}, \ldots , x_{n} \in X $, the following matrix in $M_{2n}(\mathbb{C})$
		$$ \left[\begin{array}{cc}
			\quad [K_{11}(x_{\mu}, x_{\nu})]_{\mu, \nu=1}^{n}  &  \quad [K_{12}(x_{\mu}, x_{\nu})]_{\mu,\nu=1}^{n} \\ \quad   [ K_{21}(x_{\mu}, x_{\nu}) ]_{\mu, \nu=1}^{n} & \quad  [K_{22}(x_{\mu}, x_{\nu})]_{\mu, \nu=1}^{n} 
		\end{array} \right], \quad
		$$	
		is positive semidefinite (definite).
	\end{lem}

	We denote by $Z(S)$ the center of the semigroup $S$, that is the abelian semigroup
	$$
	Z(S):=\{\phi \in S : \phi \psi = \psi \phi, \text{ for all } \psi \in S  \}.
	$$
	Note that $Z(S)$ is never empty, being the identity function $i : X \to X$ an example.
	
	\begin{thm}\label{construcao} Let $k \in P(X,S)$ and $\phi \in Z(S)$. The matrix valued kernel $K: X \times X \to M_{2}(\mathbb{C})$ given by
		$$
		K(x,y)=\left[K_{\mu\nu}\right]_{\mu,\nu=1}^2=\left[\begin{array}{cc}
			k(\phi(x),\phi(y)) & k(\phi(x),y) \\
			k(x, \phi(y)) & k(x,y)
		\end{array}\right], \quad x, y \in X,
		$$	
		belongs to $P(X,S,\mathbb{C}^{2})\setminus P^{+}(X,S,\mathbb{C}^{2})$. Also, if $P^{+}(X,S)$ is non empty then  all functions in $S$  are injective.
	\end{thm}
	\begin{proof}The kernel $K$ is continuous because both the kernel $k$ and the function $\phi$ are continuous. The kernel $K$ is invariant by the semigroup $S$ because if $x,y \in X$ and $\psi \in S$, we have that
		$$
		K_{12}(\psi(x), \psi(y))= k(\phi(\psi(x)), \psi(y))=k(\psi(\phi(x)), \psi(y))=k(\phi(x), y)= K_{12}(x, y).
		$$ 	
		In the second equality we have used the fact that $\phi \in Z(S)$ and in the third one that $k$ is invariant by the semigroup $S$. Similarly for the   other 3 entries of $K$.\ Next, let us verify the positive definiteness of $K$. For distinct points $x_{1}, \ldots , x_{n} \in X$, we have that
		$$
		\left [ \begin{array}{cc}
			\left[k( \phi(x_{\mu}),  \phi(x_{\nu}) )\right]_{\mu,\nu=1}^n &  \left[k( \phi(x_{\mu}),   x_{\nu})\right]_{\mu,\nu=1}^n \\
			\left[k( x_{\mu},  \phi(x_{\nu}))\right]_{\mu,\nu=1}^n &  \left[k( x_{\mu},  x_{\nu})\right]_{\mu,\nu=1}^n
		\end{array}\right ]=\left[k( y_{i}, y_{j})\right]_{i,j=1}^{2n},
		$$
		where
		$$
		y_{i}=\left\{\begin{array}{ll}
			\phi(x_{i})  & \mbox{if $i=1,\ldots,n$}\\
			x_{i-n} & \mbox{if $i=n+1,\ldots,2n$}
		\end{array} \right.
		$$
  
		Thus, the positive definiteness of $K$ follows from that of $k$. In order to see that $K$ is not strictly positive definite,
		we consider two cases. Given an arbitrary $x \in X$, either  $ \phi(x)=x $, and then
		$$K( x,  x)=\left[\begin{array}{cc}
			k( x,  x) & k( x,  x) \\
			k( x,  x) & k( x,  x)
		\end{array}\right]
		$$
		and $K$ is obviously not strictly positive definite, or  $\phi(x) \neq x $, defining $x_{1}=x$ and $x_{2}= \phi(x)$, we have that
		$$
		\left[K( x_{\mu}, x_{\nu})\right]_{\mu,\nu=1}^2 \simeq \left[\begin{array}{cccc}
			k(\phi(x),\phi(x)) & k(\phi(x), \phi^{2}(x)) & k(\phi(x),x) & k(\phi(x),\phi(x) \\
			k(\phi^{2}(x),\phi(x) ) & k(\phi^{2}(x) , \phi^{2}(x)) & k(\phi^{2}(x), x) & k(\phi^{2}(x),\phi(x)) \\
			k(x,\phi(x)) & k(x,\phi^{2}(x)) & k(x,x) & k(x,\phi(x)) \\
			K(\phi(x),\phi(x)) & k(\phi(x),\phi^{2}(x)) & k(\phi(x),x) & k(\phi(x),\phi(x))
		\end{array}\right]
		$$
		is a matrix of rank $<4$.\\
		For the second part, if by an absurd there is a  $k \in P^{+}(X,S, \mathbb{C})$ but there is a $\phi \in S$ that is non injective, taking distinct $x_{1}, x_{2} \in X$ such that $\phi(x_{1})= \phi(x_{2})=z$, the matrix
		$$
		\left[k( x_{\mu}, x_{\nu})\right]_{\mu,\nu=1}^{2} = \left[\begin{array}{cc}
			k(x_{1},x_{1}) & k(x_{1}, x_{2})\\
			k(x_{2},x_{1}) & k(x_{2}, x_{2}) 
		\end{array}\right] = 
		\left[\begin{array}{cc}
			k(z,z) & k(z,z)\\
			k(z,z) & k(z,z) 
		\end{array}\right]
		$$
		is non invertible, which is an absurd.	
	\end{proof}

	To proceed we introduce a new definition, which often appears in dynamical system theory.
	
	\begin{defn} Let $X$ be a topological space. A function $\phi$ in a semigroup $S$ of functions in $C(X,X)$, is said to  be  aperiodic if  $\phi^{m}(x) \neq x$ for all $x \in X$ and  $m\in \mathbb{N}$.  
	\end{defn}
	
	Note that the existence of an aperiodic function is more restrictive than demanding that the semigroup $S$ is non torsion.

	The next lemma and its corollary will help us understand how the set 
	$$
	\{y_{1}, y_{2}, \ldots , y_{2n} \}:=\{\phi(x_{1}), \phi(x_{2}), \ldots \phi(x_{n})  ,x_{1}, x_{2}, \ldots x_{n}   \}.
	$$
	that appeared on the proof of Theorem \ref{construcao} behaves if $\phi$ is an aperiodic injective function.

	\begin{lem}\label{quantidade}  Let $\phi \in C(X,X)$ be an aperiodic injective function. For any distinct points  $x_{1}, x_{2}, \ldots x_{n} \in X$, consider the set  $F:=\{ \mu \in  \{1,2,\ldots , n\}, \quad \phi (x_{\mu}) \in \{x_{1}, x_{2}, \ldots x_{n}\}  \} $ and the function $\tau: F\to \{1,2,\ldots , n\}$ for which $x_{\tau(\mu)}= \phi (x_{\mu})$. Then the set $F$ and the function  $\tau $ are well defined and satisfy :\\
		$(i)$ $\tau$ is injective\\
		$(ii)$ $0\leq |F|\leq n-1$.\\	
		$(iii)$ If $\mu \in F$ there exists $q\in \mathbb{N}$, which depends on $\mu$, where $\mu, \tau(\mu), \tau^{2}(\mu), \dots \tau^{q}(\mu) \in F$, but $\tau^{q+1}(\mu)\notin F$.
		
	\end{lem}
	
	\begin{proof} The function $\tau$ is well defined because the points $x_{1}, \ldots, x_{n}$ are distinct and is injective because $\phi$ is an injective function.\\
		In order to prove $(ii)$, suppose by contradiction that $|F|=n$. In particular, the function $\tau:\{1,2,\ldots , n\} \to \{1,2,\ldots , n\} $ is bijective. Since the group of bijective functions over $\{1,2,\ldots , n\}$ is finite, there will be an $M \in \mathbb{N}$ where $\tau^{M} =I_{n}$, the identity function over $\{1,2,\ldots , n\}$. Note that $M \neq 1$ because $\phi$ is aperiodic, so $M\geq 2$, but 
		$$ x_{1}= x_{\tau^{M}(1)}= \phi(x_{\tau^{M-1}(1)})= \ldots =  \phi^{M-1}(x_{\tau}(1))= \phi^{M}(x_{1})$$
		but then $\phi^{M}(x_{1})= x_{1}$, which is an absurd because $\phi$ is aperiodic.\\
		The proof of $(iii)$ is similar to the one we presented in $(ii)$. Suppose by contradiction that exists $\mu \in F$ for which  $\tau^{q}(\mu) \in F$ for all $q \in N$. Since the set $\{x_{1}, \ldots , x_{n}\}$ is finite, there will exist natural numbers $p < q$, where $x_{\tau^{p}(\mu)}= x_{\tau^{q}(\mu)}$. But,  
		$$
		\phi^{q-p}(x_{\tau^{p}(\mu)})= \phi^{q-p}(\phi^{p}(x_{\mu}))=  \phi^{q}(x_{\mu})= x_{\tau^{q}(\mu)}=x_{\tau^{p}}(\mu)
		$$ 
		which is an absurd because $\phi$ is aperiodic.	
		
	\end{proof}

	We restate Lemma \ref{quantidade} in a more friendly way for the proof of our main result. The proof is omitted. 	
	
	\begin{cor}\label{quantidade2} Under the same hypotheses of Lemma \ref{quantidade}, there exists $m \in  \mathbb{Z}_{+}$ and $p \in \mathbb{N}$, with $m=|F|$ and $p=n-|F|$, $m+2p$ distinct points $ z_{\alpha} \in X$, such that:\\
		$(i)$ $\{\phi(x_{1}), \phi(x_{2}), \ldots \phi(x_{n})  ,x_{1}, x_{2}, \ldots x_{n}   \}=\{  z_{1}, \ldots, z_{m},  z_{m+1},\ldots,   z_{m+p},  z_{m+p+1},\ldots,  z_{m+2p} \}$.\\
		$(ii)$  $\{ z_{1},\ldots,  z_{m} \}:= \{ \phi(x_{\mu}), \mu \in F  \} = \{ x_{\tau(\mu)}, \mu \in F  \} $.\\
		$(iii)$ $\{ z_{m+1},\ldots,  z_{m+p} \}:= \{ \phi(x_{\eta}), \eta \notin F   \} $.\\
		$(iv)$  $\{ z_{m+p+1},\ldots,  z_{m+2p} \}:= \{ x_{\mu}, \mu \notin \tau(F)  \} $.
	\end{cor}
	
	We just emphasize that the sets in $(iii)$ and $(iv)$ have the same number of elements because $\tau$ is an injective function. 
	
	Now we are able to prove the main result of this paper.

	\begin{thm} \label{principal}   Suppose that $S$ is a semigroup of functions in $C(X,X)$,  $\phi \in Z(S)$ is an aperiodic function and that  $P^{+}(X,S)$ is non empty. Then, for any $\ell \geq 2$
		$$
		P^{+}(X,S, \mathbb{C}^{\ell}) \subsetneq  P^{+}_{Proj}(X,S, \mathbb{C}^{\ell})
		$$
	\end{thm}	
	
	\begin{proof}By Lemma \ref{just2}, we only have to show an example for $\ell =2$.\\  Let $k \in P^{+}(X,S)$ and  $\phi\in Z(S)$ for which $\phi^{m}(x)\neq x $ for all $m \in \mathbb{N}$ and $x \in X$. Theorem \ref{construcao}
		asserts that the kernel $K: X \times X \to M_{2}(\mathbb{C})$ given by
		$$
		K( x, y)=\left[\begin{array}{cc}
			k( \phi(x), \phi(y)) & k( \phi(x), y) \\
			k( x,  \phi(y)) & k( x, y)
		\end{array}\right], \quad  x,  y \in X,
		$$	
		belongs to $P(X,S,\mathbb{C}^{2})\setminus P^{+}(X,S,\mathbb{C}^{2})$ and that $\phi$ is injective. Next, we show that $K \in 	P^{+}_{Proj}(X,S, \mathbb{C}^{2})$.
		Fix $v \in \mathbb{C}^2 \setminus\{(0,0)\}$ and arbitrary distinct points $ x_{1}, \ldots,  x_{n}$ in $X$.\ Let $ y_{1},\ldots,  y_{2n}$ as in the proof of Theorem \ref{construcao}.\ We will show $[K_v( x_{\mu}, x_{\nu})]_{\mu,\nu=1}^n$ is positive definite.\ If $c_1, \ldots, c_n \in \mathbb{C}$ and 
		$$
		0=\sum_{\mu,\nu=1}^n c_\mu\overline{c_\nu}K_v( x_{\mu}, x_{\nu}),
		$$
		then,  we have that
		\begin{align*}
  &\sum_{\mu,\nu=1}^{n} c_\mu\overline{c_\nu}[v_{1}\overline{v_{1}}k( \phi(x_{\mu}),\phi(  x_{\nu})) + v_{1}\overline{v_{2}}k( \phi(x_{\mu}),  x_{\nu})+v_{2}\overline{v_{1}}k( x_{\mu},  \phi(x_{\nu}))+v_{2}\overline{v_{2}}k( x_{\mu},  x_{\nu}) ]\\
  &=\sum_{i,j=1}^{2n} d_{i}\overline{d_{j}}k(y_{i}, y_{j})=0
		\end{align*}
		where $d_{i} = v_{1}c_{i}$ and $d_{i+n}= v_{2}c_{i}$ for $1\leq i \leq n$. Taking into account Lemma \ref{quantidade} and Corollary \ref{quantidade2}, we can rewrite this double sum as:\\

		$$ 
		\sum_{\alpha,\beta=1}^{m+2p} e_{\alpha}\overline{e_{\beta}}k( z_{\alpha},  z_{\beta})=0,
		$$
		where: \\
		$(I)$ $1 \leq\alpha \leq m $,  $e_{\alpha} = v_{1}c_{\mu} + v_{2}c_{\tau(\mu)}$, for some  $\mu \in F$.\\
		$(II)$ $m+1 \leq\alpha \leq m+p $,  $e_{\alpha} = v_{1}c_{\mu} $, for some  $\mu \notin F$.\\
		$(III)$  $m+p+1 \leq\alpha \leq m +2p$,  $e_{\alpha} = v_{2}c_{\mu}$, for some  $\mu \notin \tau( F)$.
		
		But the kernel $k$ is strict and the $m+2p$ points $z_{\alpha}$ are distinct, so $e_{\alpha}=0$ for all $1\leq \alpha \leq m+2p$. We divide the proof in the case that $v_{1} \neq 0$ and the case that $v_{1} =0$.\\
		If $v_{1} \neq 0$, since $p\geq 1$, equation $(II)$ implies that $c_{\mu}=0$ for all $\mu \notin F$. If $\mu \in F$, by the relation $(iv)$ in Lemma \ref{quantidade}, there exist $q \in \mathbb{Z}_{+}$, which depends on $\mu$, for which  $\mu, \tau(\mu), \tau^{2}(\mu), \dots \tau^{q}(\mu) \in F$, but $\tau^{q+1}\notin F$. In particular, equation $(I)$ implies that $$
		0=v_{1}c_{\tau^{q}(\mu)}+v_{2}c_{\tau^{q+1}(\mu)}=v_{1}c_{\tau^{q}(\mu)},
		$$ so $c_{\tau^{q}(\mu)}=0$. In case $q \geq 1$, we use equation $(I)$ again in order to obtain 
		$$
		0=v_{1}c_{\tau^{q-1}(\mu)}+v_{2}c_{\tau^{q}(\mu)}=v_{1}c_{\tau^{q-1}(\mu)},
		$$
		so $c_{\tau^{q-1}(\mu)}=0$. After finitely many similar steps we conclude that $c_{\mu}=0$, and so, the kernel $K_{v}$ is strict.\\
		If $v_{1}=0$, then $v_{ 2} \neq 0$. Equation $(I)$ imply that $0=v_{2}c_{\eta}$, for all $\eta\in  \tau(F)$ while equation $(III)$ imply that  $\eta \notin \tau(F)$, and again, $c_{\eta} =0$ for all $\eta$ and the kernel $K_{v}$ is strict.   
	\end{proof}

	\begin{rem}If the function $\phi\in Z(S)$ in Theorem \ref{construcao}  is not aperiodic or is not injective, then it can be proved that $K \notin 	P^{+}_{Proj}(X,S, \mathbb{C}^{2})$.  
	\end{rem}

	The fact that $\phi \in Z(S)$ is an essential requirement. This can be seen in a very familiar example, the set of kernels defined on $\mathbb{R}$ and invariant by the group of functions $\psi_{i,w}(x):=(-1)^{i}x +w$, $i \in \{0,1\}$ and $w \in \mathbb{R}$, which by the comments made at the introduction, are the radial kernels in $\mathbb{R}^{1}$. The center of this group only contains the identity function $\psi_{0,0}$. Furthermore, fixed a function $\psi_{i,w} \in S$, $w \neq 0$, if a kernel $k:\mathbb{R}\times \mathbb{R} \to \mathbb{C} \in P(\mathbb{R},S)$ is such that
	$$ 
	k(x, \psi_{i,w}(y) )= k(\phi(x), \psi_{i,w}(\phi(y)) )
	$$
	for all $\phi \in S$ and $x,y \in \mathbb{R}$ (hence the kernel in Theorem \ref{construcao} would be well defined even if $\psi_{i,w}$ is not in $Z(S)$), then this  kernel is a  nonnegative constant function.

	\subsection{Abelian groups}

	In this subsection we completely characterize the locally compact Abelian groups $G$ together with the group of functions $S=\{ \phi_{g}:G \to G$, $\phi(x)=gx$, $g \in G\}$, for which $P_{proj}^{+}(G,S, \mathbb{C}^{\ell})= P^{+}(G,S,  \mathbb{C}^{\ell} ) $, for  any $\ell \in \mathbb{N}$. Those kernels are usually called in the literature as translational invariant. It is worth mentioning that  $P_{proj}(G,S, \mathbb{C}^{\ell})= P(G,S, \mathbb{C}^{\ell} ) $, for any $\ell \in \mathbb{N}$ and every locally compact Abelian group $G$, this result   is a consequence of  \cite{neb}.  
	
	Except for the finite Abelian groups, we will not need  the complete description of the positive definite kernels $P(G,S )$ or the harmonic analysis on locally compact Abelian groups. The interested reader may look at \cite{rudin} or \cite{van2009introduction} for more information on them.

	In order to prove the characterization, first note that all functions $\phi_{g} \in S$ are injective and $N(S)=S$. Moreover, $\phi^{m}_{g}(x)= g^{m}x$, so it exists a aperiodic  function  $\phi_{g}$ in  $S$  if and only if the group $G$ is non torsion. In order to use a notation more common in the literature to present this type of kernel, we omit the $S$ term. 
	
	A direct application of Theorem \ref{principal} leads to the following result:    
	
	\begin{cor} \label{principalabelian} Let $G$ be an Abelian locally compact group that is non torsion and for which $P(G, \mathbb{C})$ is non empty. Then
		$$
		P^{+}_{Proj}(G, \mathbb{C}^{\ell}) \subsetneq P^{+}(G, \mathbb{C}^{\ell}),
		$$
		for any $\ell \geq 2$.
	\end{cor}

	-The $d$-dimensional torus $T_d$, defined as
	$$
	T_d:=\{x\in \mathbb{R}^{d}: -\pi \leq x_j < \pi; j=1,2,\ldots, d\},
	$$
	is a locally compact Abelian group and  is non torsion. There exist strictly positive definite  translational invariant kernels on it like
	$$
	(x,y) \in T_d \times T_d \to \prod_{m=1}^{d} \frac{2}{2-e^{i(x_{m}-y_{m})}}= \sum_{\alpha \in \mathbb{Z}^{d}_{+}}\frac{1}{2^{|\alpha|}}e^{-i\alpha x}\overline{e^{-i\alpha y}} \in \mathbb{C}. 
	$$
	The complete characterization of $P(T_d, \mathbb{C})$ can be found in \cite{shapiro}, while the characterization for $P^{+}(T_d, \mathbb{C})$ was given in \cite{emonds2011strictly}.

	In \cite{emonds2011strictly}, it was also proved a very interesting result, if $G$ is compact then  $P^{+}(G , \mathbb{C})$  is non empty if and only if $G$ is metrizable.

	-The $d$ dimensional Euclidean space $\mathbb{R}^{d}$ is also non torsion (any nonzero element is an example), there exists a strictly positive definite translation invariant kernel on it, like the Gaussian kernel
	$$
	(x,y) \in \mathbb{R}^{d} \times \mathbb{R}^{d} \to e^{- \| x - y \|^{2} } 
	$$
	and the kernel
	$$
	(x,y) \in \mathbb{R}^{d} \times \mathbb{R}^{d} \to \left[\begin{array}{cc}
		e^{- \| x - y \|^{2} } &e^{- \| x+z - y \|^{2} } \\
		e^{- \| x - y -z \|^{2} }  & e^{- \| x - y \|^{2} } 
	\end{array}\right], \quad z \in  \mathbb{R}^{d} \setminus{\{0\}} ,
	$$	
	belongs to $P^{+}_{Proj}(\mathbb{R}^{d} , \mathbb{C}^{2}) \setminus P^{+}(\mathbb{R}^{d} , \mathbb{C}^{2})$. However, as discussed in the introduction, if we add the hypotheses that the kernel is radial and $d\geq 2$, which means that is not only invariant by translation, but also on rotations, then the two classes are actually equal.

	Now, we aim to prove the case where $G$ is a torsion group. In that case, given a finite quantity of distinct points $x_{1}, \ldots , x_{n} \in G$, the group generated by those elements 
	$$
	\langle x_{1}, \ldots x_{n} \rangle=\{   x_{1}^{p_{1}}x_{2}^{p_{2}} \ldots x_{n}^{p_{n}}, \quad p_{1}, \ldots p_{n} \in \mathbb{N} \}
	$$
	is finite and Abelian. Given a kernel $K : G \times G \to  M_{\ell}(\mathbb{C})$, it is strictly positive definite if and only if the kernel $K$ restricted to the set $\langle x_{1}, \ldots x_{n} \rangle$ is strictly positive definite for every finite quantity of distinct points $x_{1}, \ldots x_{n} \in G$. But, by the fundamental Theorem of finite Abelian groups, a set like $\langle x_{1}, \ldots x_{n} \rangle$  is isomorphic to a direct sum of $\mathbb{Z}_{q}$ groups, that means
	$$
	\langle x_{1}, \ldots x_{n} \rangle \simeq \mathbb{Z}_{q_{1}}\times \ldots \times \mathbb{Z}_{q_{l}}.
	$$
	
	Now we prove a version of Corollary \ref{principalabelian} to a group like  $G=\mathbb{Z}_{q_{1}}\times \ldots \times \mathbb{Z}_{q_{l}}$, it turns out to be a completely different result, and the general case will follow from this case by the previous comments. Before that, we recall the characterizations of $P(G, \mathbb{C})$ and $P^{+}(G, \mathbb{C})$, which can be found in   \cite{emonds2011strictly}.

	\begin{thm}\label{audrey}  Let $G=\mathbb{Z}_{q_{1}}\times \ldots \times \mathbb{Z}_{q_{l}}$ and $\psi: G \to \mathbb{C}$. The kernel $K(x,y):= \psi(xy^{-1})$ is positive definite if and only if
		$$	
		\psi(x)= \sum_{g \in G} a_{g}\xi_{g}(x) ,
		$$	
		where $a_{g} \geq 0$  and  $\xi_{g}(x)= \prod_{r=1}^{l} e^{2\pi i g_{r}x_{r}/q_{r}}$ for all $g \in G$. The representation is unique, moreover, the kernel is strictly positive definite if and only if $a_{g} >0$ for all $g \in  G$.
	\end{thm}

	\begin{lem} \label{finitealabelian} Let $q_{1}, \ldots , q_{l} \in \{2,3,\ldots \}$, and  $G:= \mathbb{Z}_{q_{1}}\times \ldots \times \mathbb{Z}_{q_{l}}$. Then $P^{+}(G, \mathbb{C})$ is non empty and 
		$$
		P^{+}_{Proj}(G, \mathbb{C}^{\ell})= P^{+}(G, \mathbb{C}^{\ell}),
		$$
		for any $\ell \geq2$.
	\end{lem}	  
	
	\begin{proof}
		By Lemma \ref{just2}, we only need to prove the matrix valued case.  Let $\psi: G \to \mathbb{C}$ be such that the  kernel $K(x,y):= \psi(xy^{-1})$ is an element of $ P^{+}_{Proj}(G, \mathbb{C}^{\ell})$, $\ell \in \mathbb{N}$. By the results proved in \cite{neb} we have that $K \in  P(G, \mathbb{C}^{\ell}) $  and it has the following decomposition
		$$
		\psi(xy^{-1})=K(x,y)= \sum_{g \in G}A_{g}(K)\xi_{g}(x) \overline{\xi_{g}(y)} 
		$$
		where $A_{g}(K) \in M_{\ell}(\mathbb{C})$ is a positive semidefinite matrix. But, since $K \in P_{Proj}^{+}(G, \mathbb{C}^{\ell})$, for every $v \in \mathbb{C}^{\ell} \setminus{\{0\}}$, the kernel 
		$$
		K_{v}(x,y)= \langle [ \sum_{g \in G}A_{g}(K)\xi_{g}(x) \overline{\xi_{g}(y)} ] v, v \rangle = \sum_{g \in G}\langle A_{g}(K) v,v \rangle \xi_{g}(x) \overline{\xi_{g}(y)},  
		$$
		is strictly positive definite. Theorem \ref{audrey} implies that the matrices $A_{g}(K) \in M_{\ell}(\mathbb{C})$ are positive definite. Hence, if $x_{1}, \ldots, x_{n} \in G$ are distinct and  $v_{1}, \ldots, v_{n} \in \in \mathbb{C}^{\ell} $, are such that
		\begin{align*}
			0=\sum_{\mu, \nu =1}^{n} \langle K(x_{\mu},x_{\nu})v_{\mu},  v_{\nu} \rangle  &=   \sum_{\mu, \nu =1}^{n} \sum_{g \in G}\langle A_{g}(K)v_{\mu }, v_{\nu}\rangle \xi_{g}(x_{\mu}) \overline{\xi_{g}(x_{\nu})}\\
			&=  \sum_{g \in G} \langle A_{g}(K) \sum_{ \mu=1}^{n} v_{\mu}\xi_{g}(x_{\mu}) , \sum_{ \nu =1}^{n} v_{\nu}\xi_{g}(x_{\nu}) \rangle 
		\end{align*}
		then $\sum_{ \mu=1}^{n} v_{\mu}\xi_{g}(x_{\mu}) =0$, for all $g\in G$. After introducing coordinates we obtain that $v_{\mu}=0$ for all $\mu$, and then $K \in P^{+}(G, \mathbb{C}^{\ell})$.
	\end{proof}

	\begin{cor} \label{finitelabelian2} Let $G$ be an Abelian locally compact torsion group  for which $P^{+}(G, \mathbb{C})$ is non empty. Then
		$$
		P^{+}_{Proj}(G, \mathbb{C}^{\ell})= P^{+}(G, \mathbb{C}^{\ell}),
		$$
		for  any $\ell \in \mathbb{N}$.
	\end{cor}

	\subsection{Additional examples}
	
	Another class of kernels for which Theorem \ref{principal} can be applied are the  
	isotropic kernels in complex spheres. Let  $X=\Omega^q$, be the unit  complex sphere in $\mathbb{C}^{q}$, the kernel $K$ continuous and also isotropic in the sense that
	$$
	K(Qx,Qy)=K(x,y), \quad x,y \in \Omega^q; Q\in \mathcal{U}(q),
	$$
	where $\mathcal{U}(q)$ is the set of all unitary transformations on $\mathbb{C}^{q}$.\ The characterization of the positive definite scalar valued  kernels fulfilling this condition was achieved in  \cite{pdesferacomplexa}, while the strict case was proved in \cite{unitarily} ($q\geq 2$)  and \cite{spds1} ($q=1$). Similar to the real sphere, 
	there is also the limiting case, of kernels defined on $\Omega^{\infty}$,  the unit complex sphere of an infinity dimensional complex  Hilbert space, $K$ is continuous and the isotropy is the invariance of the kernel for all linear isometries in it and can be found in \cite{esferacomplexainfinita}. The characterization of the positive definite matrix valued case is a consequence of the results in Section $8$ at \cite{BERG2018259}.

	We can apply Theorem \ref{principal} for this type of kernel, because the unitary matrix $e^{i\theta}I$, where $\theta \notin \mathbb{Q}\pi$ and $I$ is the identity matrix in $C^{q}$, defines an aperiodic function in $\Omega^{q}$ that is in the center  of $\mathcal{U}(q)$.  
	
	\section{Adjointly invariant kernels}\label{Adjointly invariant kernels}
	
	In this section  we prove an analogous of Theorem \ref{principal}  to a class of kernels with a different type of symmetry. Due to the similarities in the arguments, we sometimes omit them.
	
	We recall that as involution on an arbitrary set $Z$ is a function $*: Z\to Z$ for which $*(*(z))=z$, for every $z\in Z$, and we  write $z^{*}$ instead of $*(z)$. 
	\begin{defn}\label{unitaril}Let $X$ be a topological space and a semigroup of continuous functions $S\subset C(X, X)$. We say that the kernel $K$ is adjointly invariant by the semigroup $S$ if there exists an involution function $* : S \to S$, for which
		$$
		K(x,\phi(y))= K(\phi^{*}(x), y),   
		$$
		for all  $x,y \in X$ and $\phi \in S$.
	\end{defn}
	
	Unlike Definition \ref{unitarily}, we assume from the beginning in Definition \ref{unitaril} that we have a semigroup of functions because the involution needs to be defined for all elements of the semigroup. Note that for the identity involution we could have used the approach in Definition \ref{unitarily}.

	Similar to unitarily invariant kernels, given a matrix valued   continuous kernel $K$ that  is adjointly invariant by a semigroup of continuous functions with involution $S$,  if it is positive definite we use the notation $K \in P^{*}(X,S, \mathbb{C}^{\ell})$. If in addition the kernel is strictly positive definite we use the notation $K \in P^{+,*}(X,S,\mathbb{C}^{\ell})$. Similarly, if  all of the scalar valued projections of a continuous kernel  $K$ that is unitarily invariant by a semigroup with involution of continuous functions  with $S$ are scalar valued positive definite kernels, we use the notation $K \in P^{*}_{proj}(X,S, \mathbb{C}^{\ell})$. Likewise if in addition all of the scalar valued projections of the kernel $K$ are scalar valued strictly positive definite kernels, we use the notation $K \in P^{*,+}_{proj}(X,S, \mathbb{C}^{\ell})$.

	The following Lemma is a version of Lemma \ref{just2} to the context of adjointly invariant kernels. The proof is omitted due to its similarities.	   
	
	\begin{lem}\label{just2adj}	Let $X$ be a topological space, $(S,*)$ be a semigroup of continuous functions on $X$ with an involution. Then  $P^{+, *}_{proj}(X,S,\mathbb{C}^{\ell})=P^{+, *}(X,S,\mathbb{C}^{\ell}) $ for some $\ell \geq 2$ if and only if  $P^{+, *}_{proj}(X,S,\mathbb{C}^{m})=P^{+, *}(X,S,\mathbb{C}^{m}) $  for every $2\leq m \leq \ell$.
	\end{lem}
	
	The following is a version of Theorem \ref{construcao} to the context of adjointly invariant kernels.
	
	\begin{thm}\label{constructadj} Let $k \in P^{*}(X,S,\mathbb{C})$ and $\phi \in Z(S)$. The matrix valued kernel $K: X \times X \to M_{2}(\mathbb{C})$ given by
		$$
		K(x,y)=\left[K_{\mu\nu}\right]_{\mu,\nu=1}^2=\left[\begin{array}{cc}
			k(\phi(x),\phi(y)) & k(\phi(x),y) \\
			k(x, \phi(y)) & k(x,y)
		\end{array}\right], \quad x, y \in X,
		$$	
		belongs to $P^{*}(X,S,\mathbb{C}^{2})\setminus P^{+,*}(X,S,\mathbb{C}^{2})$.  
	\end{thm}
	\begin{proof}The kernel $K$ is continuous because both the kernel $k$ and the function $\phi$ are continuous. Note that $\psi^{*}\phi= \phi\psi^{*}$ for all $\psi \in S$, so  $\psi\phi^{*}= \phi^{*}\psi$ for all $\psi \in S$ and then $\phi^{*} \in Z(S)$. The kernel $K$ is adjointly invariant by the semigroup $S$ because if $x,y \in X$ and $\psi \in S$, we have that
		\begin{align*}
			K_{11}(x, \psi(y))&= k(\phi(x), \phi(\psi(y))) =k(\phi(x), \psi(\phi(y)))=k(\psi^{*}(\phi(x)), \phi(y))\\
			&= k(\phi(\psi^{*}(x)), \phi(y))=K_{11}( \psi^{*}(x), y).
		\end{align*}
		It follows that the $K_{11}$ is adjointly invariant by the semigroup $S$. The rest of the arguments are the same as the ones in the proof of Theorem \ref{construcao}.
	\end{proof}
	
	Unlike Theorem \ref{construcao}, we do not get from the non emptiness of $P^{+,*}(X,S)$ the injectivity of all functions in $S$. 
	
	Note that in  Lemma \ref{quantidade} and Corollary \ref{quantidade2} we did not make any use of the kernel, actually we only used the properties of the functions on the semigroup, so, they are still valid in this new setting and we immediately have the proof of Theorem \ref{principal} to the context of adjointly invariant by a semigroup of functions.

	\begin{thm} \label{principal2}   Suppose that $(S,*)$ is a semigroup  with involution of functions in $C(X,X)$,  $\phi \in Z(S)$ is an aperiodic injective function and that  $P^{+,*}(X,S)$ is non empty. Then, for any $\ell \geq 2$
		$$
		P^{+,*}(X,S, \mathbb{C}^{\ell}) \subsetneq  P^{+,*}_{Proj}(X,S, \mathbb{C}^{\ell})
		$$
	\end{thm}

	\subsection{Dot product kernels}
	In this subsection we explain in detail how dot product kernels can be understood as an adjointly invariant kernel and why we had to add the identity matrix to obtain Example \ref{examplekern} for this class of kernels.
	
	\begin{lem}\label{dot1} Let $\mathcal{H}$ be a real Hilbert space and a continuous kernel $K: \mathcal{H} \times \mathcal{H} \to \mathbb{C}$. If 
		$$
		K(Ax, y)= K(x, A^{t}y), \quad x,y \in \mathcal{H}, \quad A \in \mathcal{L}(\mathcal{H}) 
		$$	
		then, there exists a continuous function $h: \mathbb{R} \to \mathbb{C}$ for which $K(x,y)=h(\langle x, y \rangle )$.	
	\end{lem}
	
	\begin{proof}
		Let $x,y$ be nonzero elements of $\mathcal{H} $. Let $(e_{\lambda})_{\lambda \in\Lambda}$ be a complete orthonormal basis for $\mathbb{H}$.\\ 
		If $x,y$ are linearly dependent, pick an  operator $A$  for which $A(e_{\lambda^{1}}):= x$ and  $A(e_{\lambda})$ is orthogonal with respect to $x$ for all $\lambda \neq \lambda^{1}$. Hence, $A^{t}(y)= \langle x, y \rangle e_{\lambda^{1}}$ and then
		$$
		K(x,y)= K(Ae_{\lambda^{1}}, y)= K(e_{\lambda^{1}}, A^{t}y)= K(e_{\lambda^{1}}, \langle x, y \rangle e_{\lambda^{1}}).
		$$
		If $x,y$ are linearly independent, pick an  operator $A$  for which $A(e_{\lambda^{1}}):= x$, $A(e_{\lambda^{2}}):=x - \langle x, y \rangle y/ \|y\|^{2}$  and $A(e_{\lambda})$ is orthogonal with respect to $x$ and $y$ for all $\lambda \neq \lambda^{1}, \lambda^{2}$. Hence, $A^{t}(y)= \langle x, y \rangle e_{\lambda^{1}}$ and then 
		$$
		K(x,y)= K(Ae_{\lambda^{1}}, y)= K(e_{\lambda^{1}}, A^{t}y)= K(e_{\lambda^{1}}, \langle x, y \rangle e_{\lambda^{1}}).
		$$
		By the continuity of $K$, we may  include the cases where either $x$ or $y$ are zero. 	
	\end{proof}

	With the exact same arguments it is possible to prove a complex version of Lemma \ref{dot1} 
	
	\begin{lem}\label{dot2} Let $\mathcal{H}$ be a complex Hilbert space and a continuous kernel $K: \mathcal{H} \times \mathcal{H} \to \mathbb{C}$. If 
		$$
		K(Ax, y)= K(x, A^{*}y), \quad x,y \in \mathcal{H} , \quad A \text{ is injective } 
		$$	
		then, there exists a continuous function $h: \mathbb{C} \to \mathbb{C}$ for which $K(x,y)=h(\langle x, y \rangle )$.	
	\end{lem}
	
	The characterization of which functions $h$ generates a positive definite kernel is obtained  in \cite{esferacomplexainfinita} when $\dim(H)= \infty$, and generalized to   $\dim_{\mathbb{C}}(H)\geq 3$ in \cite{pinkuscompl}. We are not aware of a complete characterization  for the cases $\dim_{\mathbb{C}}=1,2$.  The strict  case was also proved in \cite{pinkuscompl} for $\dim_{\mathbb{C}}(H)\geq 3$.

	Usually, the dot product kernels are presented in a different way, as the characterization of which functions $h: \mathbb{R} \to \mathbb{C}$ are such that for any positive semidefinite matrix $A$, the matrix $h(A)$ (pointwise defined) is positive semidefinite, that is, which functions preserves positivity. Along this line several important results were achieved in   \cite{momentsequence}, \cite{Belton2019} and  \cite{PositivityII}.  
	
	By taking $k$ and $\phi(x) = rx$ ($r$ is not a root of unity in the real or complex case), the counterexample in Theorem \ref{principal2}  can only be applied to the set $\mathcal{H}\setminus{\{0\}}$ ($0$ is a fixed point). If we add the point $0 \in \mathcal{H}$, the matrix valued kernel defined in Theorem  \ref{constructadj}  does not belong to $P^{+,*}_{proj}(\mathcal{H}, S, \mathbb{C}^{2})$ (the matrix $K(0,0)$ is not invertible). However, consider the matrix valued kernel  
	$$
	K(x,y) + k(0,0)I :=\left[\begin{array}{cc}
		k(rx, ry)+k(0,0) & k(rx,y) \\
		k(x, ry)  & k(x,y) +  k(0,0)
	\end{array}\right], \quad x,y \in \mathcal{H}.
	$$
	- It belongs to $P^{*}(\mathcal{H}, S, \mathbb{C}^{2})\setminus{P^{+,*}(\mathcal{H}, S, \mathbb{C}^{2})}$. The main issue is to prove that it does not belongs to $P^{+,*}(\mathcal{H}, S, \mathbb{C}^{2})$. For  that we pick $3$ points (instead of $2$ used in Theorem \ref{constructadj}), being  $0, x, rx$ and the corresponding vectors in order $(-1,1), (1,0), (0,-1)$.\\
	-If $k \in P^{+}(\mathcal{H}, S)$ and $r$ is not a root of unity (in the real or complex case), then it belongs to $P_{proj}^{+,*}(\mathcal{H}, S, \mathbb{C}^{2})$. To prove this we divide the argument into three parts:

	Part 1)  The scalar valued kernel $l(x,y):=k(x,y) - k(0,0)$ is strict in $\mathcal{H}\setminus{\{0\}}$:\\
	Just note that 
	$$
	\sum_{\mu, \nu=1}^{n}c_{\mu}\overline{c_{\nu}}l(x_{\mu},x_{\nu})=\sum_{\mu, \nu=1}^{n}c_{\mu}\overline{c_{\nu}}[ k(x_{\mu},x_{\nu}) - k(0,0)] = \sum_{\mu, \nu=0}^{n}c_{\mu}\overline{c_{\nu}}k(x_{\mu}, x_{\nu})
	$$
	where $c_{0}:= - \sum_{\mu=1}^{n}c_{\mu}$ and $x_{0}:= 0$.

	Part 2) The matrix valued kernel $L$ belongs to $P^{*}(\mathcal{H}\setminus{\{0\}}, S, \mathbb{C}^{2})\setminus{P^{+,*}(\mathcal{H}\setminus{\{0\}}, S, \mathbb{C}^{2})}$ and to $P^{+,*}_{proj}(\mathcal{H}\setminus{\{0\}}, S, \mathbb{C}^{2})$ when $r$ is not a root of unity (in the real or complex case):\\
	This occurs because  we are dealing with an aperiodic and injective function in $\mathcal{H}\setminus{\{0\}}$, so we can apply Theorem \ref{constructadj} and Theorem \ref{principal2}.

	Part 3) Conclusion:\\ 
	-Let $v= (v_{1}, v_{2})\neq (0,0)$. Then, for distinct points $x_{0}, \ldots x_{n} \in \mathcal {H}$, where $x_{0}=0$, and complex scalars $c_{0}, \ldots c_{n}$,  the scalar valued projection of the kernel $K + k(0,0)I$ in $v$  satisfies
	\begin{align*}
		\sum_{\mu, \nu=0}&c_{\mu}\overline{c_{\nu}}[K_{v}(x_{\mu}, x_{\nu}) + k(0,0)\|v\|^{2}]\\
		=\sum_{\mu, \nu=0}&c_{\mu}\overline{c_{\nu}}\left [K_{v}(x_{\mu}, x_{\nu})-k(0,0)|v_{1} + v_{2}|^{2} +k(0,0)|v_{1} + v_{2}|^{2}+ k(0,0)\|v\|^{2}\right ] \\
		= \sum_{\mu, \nu=1}&c_{\mu}\overline{c_{\nu}}L_{v}(x_{\mu}, x_{\nu})  +  \sum_{\mu, \nu=0}c_{\mu}\overline{c_{\nu}}\left [k(0,0)|v_{1} + v_{2}|^{2}+ k(0,0)\|v\|^{2} \right ] \\
		= \sum_{\mu, \nu=1}&c_{\mu}\overline{c_{\nu}}L_{v}(x_{\mu}, x_{\nu}) +  \left |\sum_{\mu=0}c_{\mu} \right |^{2}\left [k(0,0)|v_{1} + v_{2}|^{2}+ k(0,0)\|v\|^{2}\right ]. 
	\end{align*}
	Note that both terms are nonnegative. Since the first term does not include the element $0$, Part $2$ implies that it is zero if and only if all scalars $c_{1}, \ldots, c_{n}$ are zero. On the other hand, if all scalars  $c_{1}, \ldots, c_{n}$ are zero the second term implies that $c_{0}=0$, which concludes that $K + k(0,0)I$ is in $P_{proj}^{+,*}(\mathcal{H}, S, \mathbb{C}^{2})$.

	\begin{rem} The previous result can be generalized to an abstract setting with the same type of argument. Suppose that $\tilde{X}= X\cup \{0\}$ is a topological space for which $0$ is an accumulation point. Let $(S, *)$ is a semigroup of functions in $C(\tilde{X}, \tilde{X})$ with the identity for which $\phi(0)=0$ for any $\phi \in S$ and for every $x \in X$ its orbit satisfies $\{\phi(x), \phi \in S\}= X$ (those two properties implies that if $k \in P^{*}(\tilde{X},S)$ then $k(0,0)=k(0,x)$ for any $x \in \tilde{X}$).\\
		Hence, for any $k \in P^{*}(\tilde{X},S)$ and $\phi \in S$ the matrix valued kernel
		$$
		(x,y) \in \tilde{X}\times \tilde{X}\to  \left[\begin{array}{cc}
			k(\phi(x), \phi(y))+k(0,0) & k(\phi(x), y)  \\
			k(x,\phi(y))  & k(x,y) +  k(0,0)
		\end{array}\right], 
		$$
		belongs to $P^{*}(\tilde{X}, S, \mathbb{C}^{2})\setminus{P^{+,*}(\tilde{X}, S, \mathbb{C}^{2})}$. In addition, if  $k \in P^{+,*}(\tilde{X},S)$ and $\phi$ is an aperiodic injective  function when restricted to $X$, the previous kernel is an element of $P_{proj}^{+,*}(\tilde{X}, S, \mathbb{C}^{2})$.  	    
	\end{rem}

	Other examples of adjointly invariant  kernels are:
	
	-$X=\mathbb{R}^{m}$, $\phi_{z}(x)=x+z$ and $(\phi_{z})^{*}= \phi_{-z}$. This would lead to Bochner  (or translational invariant) kernels in $\mathbb{R}^{m}$.
	
	-$X=\mathbb{R}^{m}$, $\phi_{z}(x)=x+z$ and $(\phi_{z})^{*}= \phi_{z}$. This would lead to characterizing which functions $h: \mathbb{R}^{m} \to \mathbb{R}$ the kernel $h(x+y)$ is positive definite, which is proved in Theorem $6.5.11$ in \cite{berg0}.

\section{Acknowledgments}
Jean Carlo Guella was funded by grant 2021/04226-0, S\~ao Paulo Research Foundation
(FAPESP).

\bibliographystyle{siam} 
		\bibliography{Referrences}

\end{document}